\documentclass[a4paper,11pt]{amsart}
\usepackage[english]{babel}
\usepackage{amsmath,amssymb,amsthm}
\usepackage{calrsfs}

\usepackage[pdftex]{color}
\usepackage[dvipsnames]{xcolor}
\usepackage[bookmarks=true,hyperindex,pdftex,colorlinks, citecolor=MidnightBlue,linkcolor=MidnightBlue, urlcolor=MidnightBlue]{hyperref}

\theoremstyle{plain}
\newtheorem{theorem}{Theorem}[section]
\newtheorem{corollary}[theorem]{Corollary}

\newtheorem{lemma}[theorem]{Lemma}

\theoremstyle{definition}
\newtheorem{remark}[theorem]{Remark}
\newtheorem{example}[theorem]{Example}
\newtheorem{problem}[theorem]{Problem}

\DeclareMathOperator{\dist}{dist}

\newcommand{\R}{\mathbb{R}}
\newcommand{\N}{\mathbb{N}}
\newcommand{\D}{\mathbb{D}}
\DeclareMathOperator{\NA}{NA}

\renewcommand{\leq}{\leqslant}
\renewcommand{\geq}{\geqslant}

\begin{document}
\title[Proximinality of subspaces and lineability of norm-attaining functionals] {On proximinality of subspaces and the lineability of the set of norm-attaining functionals of Banach spaces}

\author{Miguel Mart\'{\i}n}

\address{Departamento de An\'{a}lisis Matem\'{a}tico \\ Facultad de
 Ciencias \\ Universidad de Granada \\ 18071 Granada, Spain
\newline
\href{http://orcid.org/0000-0003-4502-798X}{ORCID: \texttt{0000-0003-4502-798X} }
 }
\email{mmartins@ugr.es}

\thanks{Partially supported by Spanish MINECO/FEDER grant MTM2015-65020-P}

\begin{abstract}
We show that for every $1<n<\infty$, there exits a Banach space $X_n$ containing proximinal subspaces of codimension $n$ but no proximinal finite codimensional subspaces of higher codimension. Moreover, the set of norm-attaining functionals of $X_n$ contains $n$-dimensional subspaces, but no subspace of higher dimension. This gives a $n$-by-$n$ version of the solutions given by Read and Rmoutil to problems of Singer and Godefroy. We also study the existence of strongly proximinal subspaces of finite codimension, showing that for every $1<n<\infty$ and $1\leqslant k <n$, there is a Banach space $X_{n,k}$ containing proximinal subspaces of finite codimension up to $n$ but not higher, and containing strongly proximinal subspaces of finite codimension up to $k$ but not higher. Finally, we deal with possible infinite-dimensional versions of the previous results, showing that there are \emph{non-separable} Banach spaces whose set of norm-attaining functionals contains infinite-dimensional separable linear subspaces but no non-separable subspaces.
\end{abstract}

\date{January 31st, 2019}

\subjclass[2010]{Primary 46B04; Secondary 46B03, 46B20}

\keywords{Norm attaining functionals, renormings of Banach spaces, lineability, proximinal subspaces}

\dedicatory{To Elvira Su\'{a}rez Bellido, \emph{in memoriam}}

\maketitle

\section{Introduction}

In the early 1970's, Ivan Singer proposed the following problem \cite[Problem~1]{Singer-Roumanian}, \cite[Problem 2.1 in p.~14]{Singer}:
\begin{verse}
Let $1<n<\infty$. Does every infinite-dimensional normed linear space\\ (or, in  particular, Banach space) contain a proximal subspace of codimension $n$?
\end{verse}

Recall that a (closed) subspace $Y$ of a Banach space $X$ is said to be \emph{proximinal} if for every $x\in X$ the set
$$
P_Y(x):=\{y\in Y\colon \|x-y\|=\dist(x,Y)\}
$$
is non-empty. A hyperplane is proximinal if and only if it is the kernel of a norm-attaining functional, so the Hanh-Banach theorem assures the existence of many proximinal hyperplanes in every Banach space. Singer's question was answered recently in the negative by the late Charles J.\ Read \cite{Read}, who showed that there is an infinite-dimensional Banach space $\mathcal{R}$ containing no proximinal subspaces of finite codimension greater than or equal to two. More spaces of this kind have been constructed in \cite{KLMW-Jussieu}. The first aim of this paper is to give an $n$-by-$n$ version of Read's counterexample: we will show that for every $1<n<\infty$, there is a Banach space $X_n$ containing proximinal subspaces of codimension $n$ but which does not contain proximinal subspaces of codimension greater than or equal to $n+1$. Observe that in this case, $X_n$ contains proximinal subspaces of codimension $k$ for every $1\leq k \leq n$ (see \cite[Proposition 3.7]{Rmoutil}, for instance). The examples can be constructed in such a way that their duals and biduals are strictly convex, and the slices of their unit balls have diameter as close to $2$ as desired. All of this is contained in our section \ref{section_1}.

There is a related problem proposed by Gilles Godefroy in 2001 about the lineability of the set of norm attaining functionals. Let us recall first some notation. For a Banach space $X$, write $X^*$ to denote its topological dual and $\NA(X)$ to denote the set of those functionals $f$ in $X^*$ which attain their norm on $X$ (i.e.\ $\|f\|=f(x)$ for some $x\in X$ with $\|x\|=1$). Godefroy's problem is stated as follows  \cite[Problem~III]{Godefroy}, \cite[Question 2.26]{Band-Godefroy}:
\begin{verse}
Does the set $\NA(X)$ contain two-dimensional subspaces for every\\ infinite-dimensional Banach space $X$?
\end{verse}

This problem has been solved in the negative by Martin Rmoutil \cite{Rmoutil} by showing that $\NA(\mathcal{R})$ contains no two-dimensional subspace. Let us observe that if $X$ contains a proximinal subspace $Y$ of finite codimension $n$, then $Y^\perp:=\{f\in X^*\colon f|_Y=0\}$ is contained in $\NA(X)$ (see \cite[Proposition~III.4]{Godefroy}), so this latter set contains a $n$-dimensional linear subspace. But the fact that $Y^\perp\subset \NA(X)$ and the one that $Y$ is proximinal are not equivalent in general, see \cite[Section 2]{InduPLMS1982} for instance. Nevertheless, Rmoutil showed that both properties are equivalent for finite-codimensional subspaces of $\mathcal{R}$, getting thus the negative solution to Godefroy's problem. We refer the reader to \cite{KLMW-Jussieu} to find more examples of Banach spaces containing no subspace of dimension two in their set of norm-attaining functionals.

The result of Rmoutil provides an example of space whose set of norm-attaining functionals is ``extremely non-lineable''.
There are many ways in which $\NA(X)$ can be ``lineable'' (i.e.\ contains infinite-dimensional linear spaces) for a non-reflexive Banach space $X$. On the one hand, before Read's and Rmoutil's results, it was not known whether for every infinite-dimensional Banach space $X$ the set $\NA(X)$ is lineable, as happens in the known cases. For instance, it is clear that $\NA(c_0)$ contains a linear subspace, a fact which extends to every Banach space with a monotone Schauder basis \cite[Theorem 3.1]{AcoAizAronGar}; it is also easy to show that $\NA(C(K))$ contains an infinite-dimensional linear space for every infinite Hausdorff compact topological space $K$ \cite[Theorem 2.1]{AcoAizAronGar}. If $X$ is isometrically isomorphic to an infinite-dimensional dual space, then it is clear that $\NA(X)$ contains a closed infinite-dimensional subspace: namely, the isometric predual of $X$. There are also some results stating when a Banach space can be equivalently renormed in such a way that the set of norm-attaining functionals for the new norm contains a linear subspace or a closed infinite-dimensional linear subspace. On the one hand, every Banach space containing an infinite-dimensional separable quotient can be renormed to get that the set of norm-attaining functionals contains an infinite-dimensional linear subspace \cite[Corollary 3.3]{GarPug}. On the other hand, a separable Asplund space $X$ can be renormed in such a way that the set of norm-attaining functionals contains a closed infinite-dimensional vector space if and only if $X^*$ contains an infinite-dimensional reflexive subspace \cite[Theorem 2.15]{Band-Godefroy}.

For every $1<n<\infty$, we provide, also in section \ref{section_1}, examples of infinite-dimensional Banach spaces $X$ for which $\NA(X)$ contains $n$-dimensional subspaces but it contains no subspace of dimension $n+1$, completing thus the picture about the magnitude of lineability which is possible for the set of norm-attaining functionals of a Banach space. This answers a question that Richard Aron possed after a talk on the contents of \cite{KLMW-Jussieu} during the \emph{Workshop on infinite dimensional analysis 2018}, held in Valencia in February 2018.

Next, we deal in section \ref{section_2} with a stronger version of proximinality introduced by V.~Indumathi and G.~Godefroy \cite{GodeInduRevMatComp}. A proximinal subspace $Y$ of a Banach space $X$ is said to be \emph{strongly proximinal} if for every $x\in X$ and every $\varepsilon>0$, there is $\delta>0$ such that $\dist(y,P_Y(x))<\varepsilon$ whenever $y\in P_Y(x,\delta):=\{y\in Y\colon \|x-y\|<\dist(x,Y)+\delta\}$. A closed hyperplane is strongly proximinal if and only if it is the kernel of a functional at which the dual norm is strongly subdifferentiable \cite[Lemma 1.1]{GodeInduRevMatComp} (see the definition of strong subdifferentiability at the end of the introduction). There are Banach spaces containing no strongly proximinal hyperplanes, as it is the case of the natural predual of the Hardy space $H^1(\D)$ by \cite[Proposition III.4.5]{D-G-Z} (see \cite[p.~117]{GodeInduRevMatComp}). We show that $\mathcal{R}$ has no strongly proximinal hyperplanes, and this also happens for other spaces constructed in \cite{KLMW-Jussieu} having no proximinal two-codimensional subspaces. We actually do not know whether there exists a Banach space $X$ having strongly proximinal hyperplanes but containing no proximinal subspaces of codimension $2$. Our main result here is to show that given $1<n<\infty$ and $1\leq k <n$, there is an infinite-dimensional Banach space $X$ containing strongly proximinal subspaces of codimension $k$ but not of codimension greater than $k$, contains proximinal subspaces of codimension $n$ but not of codimension greater than or equal to $n+1$. We do not know whether the case $k=n$ is possible.

Finally, we devote section \ref{section_3} to possible infinite-dimensional version of the previous results. After a talk on these topics at the \emph{14th ILJU School of Mathematics -- Banach Spaces and Related Topics} (Seoul, Republic of Korea, January 2019), Richard Aron asked whether it is possible to construct a non-separable Banach space $X$ such that $\NA(X)$ contains infinite-dimensional separable linear subspaces but no non-separable linear subspaces. The answer to the question is positive, and actually something more can be proved related to factor-reflexive subspaces. Recall that a subspace $Y$ of a Banach space $X$ is called \emph{factor reflexive} if $X/Y$ is reflexive. We first show that there is a \emph{non-separable} Banach space $X_\infty$ such that $\NA(X_\infty)$ contains infinite-dimensional separable subspaces but no non-separable subspaces, that $X_\infty$ contains $n$-codimensional proximinal subspaces for every $n\in \N$, but every proximinal factor reflexive subspace $X_\infty$ is of finite-codimension and it is not strongly proximinal. Besides, we construct a \emph{non-separable} Banach space $\overline{X}_\infty$ containing a  proximinal subspace $Y$ such that $\overline{X}_\infty/Y$ is infinite-dimensional, reflexive, and separable, so $\NA(\overline{X}_\infty)$ contains infinite-dimensional closed separable subspaces, but $\NA(\overline{X}_\infty)$ contains no non-separable subspaces so, in particular, every factor reflexive proximinal subspace $Y$ of $\overline{X}_\infty$ satisfies that $\overline{X}_\infty/Y$ is separable.

We would like to send the reader's attention to the paper \cite{GodeInduJAT} by G.~Godefroy and V.~Indumathi, where some interesting related problems, which are open to the best of our knowledge, are stated.

\subsection{Notation and terminology}

Let us finish this introduction with the notation and terminology which we will need along the paper. In this paper, we only deal with real Banach spaces. For a Banach space $X$, $X^*$ denotes its topological dual, $B_X$ and $S_X$ are, respectively, the closed unit ball and the unit sphere of $X$.

A Banach space $X$ (or its norm) is said to be \emph{strictly convex} if $S_X$ does not contain any non-trivial segment or, equivalently, if $\|x+y\|<2$ whenever $x,y\in B_X$, $x\neq y$. The space $X$ is said to be \emph{smooth} if its norm is G\^{a}teaux differentiable at every non-zero element. The norm of $X$ is \emph{strongly subdifferentiable} at a point $x\in X$ \cite{FranPaya} if the one-side limit
$$
\lim_{t\to 0^+} \frac{\|x+th\|-\|x\|}{t}
$$
exists uniformly for $h\in B_X$. The norm of $X$ is Fr\'{e}chet differentiable at $x$ if and only if it is both G\^{a}teaux differentiable and strongly subdifferentiable at $x$ \cite[p.~48]{FranPaya}. The norm of $X$ is said to be \emph{rough} if there is $0<\rho\leq 2$ such that
$$
\limsup_{\|h\|\to 0}\frac{\|x+h\|+\|x-h\|-2\|x\|}{\|h\|}\geq \rho
$$
for every $x\in X$, and in this case we say that the norm of $X$ is $\rho$-\emph{rough}. Observe that if the norm of $X$ is rough, then the norm of $X$ cannot be Fr\'{e}chet differentiable at any non-zero element. It is known that the norm of the dual of a Banach space is $\rho$-rough if and only if all the slices of the unit ball of the space has diameter greater than or equal to $\rho$ \cite[Proposition I.1.11]{D-G-Z}. We refer the reader to the books \cite{D-G-Z}, \cite{FHHMPZ}, and \cite{FHHMZ} for more information and background on the above notions.

For $1\leq p \leq \infty$, we write $\ell_p^{(k)}$ to denote the $k$-dimensional $\ell_p$ space. Given two Banach spaces $X_1$ and $X_2$, we write $X_1\oplus_p X_2$ to denote the $\ell_p$-sum of the spaces. It is easy to check that for $1<p<\infty$, the space $X_1\oplus_p X_2$ is strictly convex (respectively smooth) if and only if the spaces $X_1$ and $X_2$ are strictly convex (resp.\ smooth), see \cite[Excersises 5.2 and 5.39]{Megginson} for instance. The next result is also easy to check, but we have not found any reference, so we state it here and include an idea of how to prove it.

\begin{remark}\label{remark-Frechet-baja}
Let $1<p<\infty$ and two Banach spaces $X_1$ and $X_2$ be given, and write $X=X_1\oplus_p X_2$. If the norm of $X$ is Fr\'{e}chet smooth at a point $(x_1,x_2)$ with $x_2\neq 0$, then the norm of $X_2$ is Fr\'{e}chet smooth at $x_2$.
\end{remark}

Indeed, this follows straightforwardly from the fact that the $p$-power of the norm of $X$ is also Fr\'{e}chet differentiable at the point $(x_1,x_2)$ and the fact that a convex function $f:Z\longrightarrow \R$ defined on a Banach space $Z$ is Fr\'{e}chet differentiable at $z\in Z$ if and only if
$$
\lim_{t\to 0} \frac{f(z+th)+f(z-th)-2f(z)}{t}=0
$$
uniformly for $h\in S_Z$ \cite[Lemma 7.4]{FHHMZ}.

The next result follows easily from the definition of strong subdifferentiability (or better from \cite[Theorem 1.2.iii]{FranPaya}) and the shape of the duality function of an $\ell_1$-sum of Banach spaces.

\begin{remark}\label{remark_2}
Let two Banach spaces $X_1$ and $X_2$ be given, and write $X=X_1\oplus_1 X_2$. If the norm of $X$ is strongly subdifferentiable at a point $(x_1,x_2)$ with $x_2\neq 0$, then the norm of $X_2$ is strongly subdifferentiable at $x_2/\|x_2\|$.
\end{remark}

\section{The $n$-by-$n$ version of Read's and Rmoutil's examples}\label{section_1}
Let us start by stating the main result of the paper. We write it in such a way that solves Singer's and Godefroy's problems at the same time for a fixed $n$.

\begin{theorem}\label{theorem:main}
Let $1\leq n<\infty$. Then there is an infinite-dimensional Banach space $X_n$ such that
\begin{itemize}
  \item[(a)] $X_n$ contains $n$-codimensional proximinal subspaces,
  \item[(b)] but it does not contain $m$-codimensional proximinal subspaces for $m\geq n+1$;
  \item[(c)] $\NA(X_n)$ contains $n$-dimensional subspaces,
  \item[(d)] but it does not contain subspaces of dimension greater than $n$.
\end{itemize}
Moreover, for every $0<\varepsilon<2$, such spaces $X_n$ can be constructed to be isomorphic to $c_0$ and in such a way that $X_n$ is strictly convex and smooth, $X_n^*$ is strictly convex and smooth, $X_n^{**}$ is strictly convex, and every slice of the unit ball of $X_n$ has diameter greater than or equal to $2-\varepsilon$ or, equivalently, the norm of $X_n^*$ is $(2-\varepsilon)$-rough.
\end{theorem}

We need some preliminary results. The first one deals with the set of norm-attaining functionals of an $\ell_p$-sum of spaces for $1<p\leq \infty$.

\begin{lemma}\label{lemma_NAsum}
Let $X_1$, $X_2$ be Banach spaces, let $1<p\leq \infty$, and write $X=X_1\oplus_p X_2$. Then, $\NA(X)\subseteq \NA(X_1)\times \NA(X_2)$.
\end{lemma}

\begin{proof}
Write $q=p/(p-1)$ if $1<p<\infty$ and $q=1$ if $p=\infty$. Let $f=(f_1,f_2)\in \NA(X)$ and let $x=(x_1,x_2)\in S_X$ such that $|f(x)|=\|f\|$. Then, we have that
\begin{align*}
  \|f\|=f(x)=f_1(x_1) + f_2(x_2) &\leq \|f_1\|\|x_1\| + \|f_2\|\|x_2\|  \\
   &\leq \bigl\|\bigl(\|f_1\|,\|f_2\|\bigr)\bigr\|_q\, \bigl\|\bigl(\|x_1\|,\|x_2\|\bigr)\bigr\|_p = \|f\|.
\end{align*}
Therefore, we get that
$$
f_1(x_1)=\|f_1\|\|x_1\| \quad \text{and} \quad f_2(x_2)=\|f_2\|\|x_2\|.
$$
Let us show that $f_1\in \NA(X_1)$, being the result for $f_2$ completely analogous. Indeed, if $x_1\neq 0$, then we clearly have that $f_1$ attains its norm at $x_1/\|x_1\|$. If, otherwise, $x_1=0$, then $$\|f\|=f(x)=f_2(x_2)\leq \|f_2\|\leq \|f\|,$$ so $\|f\|=\|f_2\|$. Since $q<\infty$ (as $p\neq 1$), this implies that $f_1=0$, so $f_1\in \NA(X)$ trivially.
\end{proof}

Let us comment that the above result is false for $p=1$ as, for instance, $$\NA(\ell_1^{(2)})=\bigl(\{-1,1\}\times [-1,1]\bigr)\cup \bigl([-1,1]\times \{-1,1\}\bigr).$$

The following observation is a direct consequence of Lemma \ref{lemma_NAsum}.

\begin{corollary}\label{cor-contains_less_d1+d2}
Let $X_1$, $X_2$ be Banach spaces, let $1<p\leq \infty$, and write $X=X_1\oplus_p X_2$. If $\NA(X_i)$ does not contain linear subspaces of dimension greater than $d_i$ for $i=1,2$, then $\NA(X)$ does not contain linear subspaces of dimension greater than $d_1+d_2$.
\end{corollary}

Our next preliminary result deals with the behaviour of the proximinality of subspaces and $\ell_p$-sums. All the results are well known to experts but, for the sake of completeness, we present an omnibus lemma here which also includes the results for strong proximinality we will use later on.

\begin{lemma}\label{lemma-proximinality}
Let $X_1$, $X_2$, $Z$ be Banach spaces, let $1< p\leq \infty$, and write $X=X_1\oplus_p X_2$.
\begin{itemize}
\item[(a)] If $Y_1\leq X_1$ and $Y_2\leq X_2$ are proximinal subspaces, then $Y_1\times Y_2$ is proximinal in $X$.
\item[(b)] $\{0\}\times X_2$ is strongly proximinal in $X$.
\item[(c)] If $Y$ is a factor reflexive (in particular, finite-codimensional) proximinal subspace of $Z$, then $Y^\perp\subset \NA(Z)$.
\item[(d)] If $Y$ is a finite-codimensional strongly proximinal subspace of $Z$, then the norm of $Z^*$ is strongly subdifferentiable at every norm-one element of $Y^\perp$.
\end{itemize}
\end{lemma}

Let us give an indication of how to get the above results from previous references.
Item (a) is contained in \cite[Corollary 4.2]{BandLiLinNarayana} for $1< p<\infty$ and in \cite[Remark 3.1]{IndumathiPAMS2005} for $p=\infty$, for instance. Item (b) is easily computable for $1<p<\infty$ and also follows from \cite[Proposition 3.1]{JayaPaul2015} or by \cite[Theorem 4.1]{BandLiLinNarayana}; for $p=\infty$, the argument is not as direct as in the previous case but follows from \cite[Proposition 3.2]{JayaPaul2015} or by \cite[Theorem 2.6]{LatiNaraTJM2006}. (c) can be found in \cite[Lemma 2.2]{Band-Godefroy}, for instance. Finally, item (d) appeared without proof in \cite[Theorem 1.4]{JayaPaul2015}, referring to \cite{GodeInduRevMatComp}; the result does not actually appear in \cite{GodeInduRevMatComp}, but follows from the results of that paper by an argument that professor V.~Indumathi gently gave to us: consider a norm-one element in $Y^\perp$, extend it to a basis of the finite-dimensional subspace $Y^\perp$, and observe that the conditions in \cite[Theorem 2.5]{GodeInduRevMatComp} give the desired result by using \cite[Lemma 1.1]{GodeInduRevMatComp}.

We are now able to prove the main result.

\begin{proof}[Proof of Theorem~\ref{theorem:main}] We give an induction argument. For $n=1$, what we need are just the results of Read and Rmoutil \cite{Read,Rmoutil} cited in the introduction, so we may consider $X_1$ any infinite-dimensional Banach space having no two-dimensional subspaces in its set of norm-attaining functionals as, for instance, $\mathcal{R}$. Suppose that $X_n$ satisfies item (a)--(d) of the theorem. Given $1<p\leq \infty$ fixed, we define $X_{n+1}=X_n \oplus_p X_1$. Then, taking a proximinal subspace $Y_n$ of $X_n$ of codimension $n$ and a proximinal hyperplane $Y_1$ of $X_1$, $Y_n\times Y_1$ is proximinal in $X_{n+1}$ by  Lemma \ref{lemma-proximinality}.a. and, clearly, its codimension is equal to $n+1$. This gives (a), and (c) is a consequence of (a) as $[Y_n\times Y_1]^\perp \subset \NA(X_{n+1})$ by Lemma \ref{lemma-proximinality}.c. On the other hand, Corollary \ref{cor-contains_less_d1+d2} shows that $\NA(X_{n+1})$ does not contain linear subspaces of dimension greater than $n+1$, giving (d). Then, another use of Lemma \ref{lemma-proximinality}.c gives that $X_{n+1}$ does not contain $m$-codimensional proximinal subspaces for $m>n+1$, providing thus (b).

For the moreover part, it is enough to consider $X_1$ a renorming of $c_0$ given in \cite[Proposition 4.7]{KLMW-Jussieu} having these properties, and observe that all the properties are preserved by $\ell_p$-sums for $1<p<\infty$. This is classical and well-known for smoothness and strict convexity as we commented at the end of the introduction. For the size of the diameter of the slices of the unit ball (and therefore for the roughness of the dual norm by \cite[Proposition I.1.11]{D-G-Z}), the result appeared in \cite[Theorem 2.4]{AcoBecLopez} for diameter $2$, but the proof adapts straightforwardly to any other size of the diameter.
\end{proof}

It is shown in \cite{KLMW-Jussieu} that every Banach space which is isomorphic to a subspace of $\ell_\infty$ and contains $c_0$ can be equivalently renormed so that it does not contain proximinal subspaces of codimension two. We can use this fact to get the following result.

\begin{remark}
{\slshape Let $X$ be a Banach space which is isomorphic to a subspace of $\ell_\infty$ and contains $c_0$. Then for $1 < n<\infty$, there is a Banach space $X_n$ isomorphic to $X$ such that:
\begin{itemize}
  \item[(a)] $X_n$ contains $n$-codimensional proximinal subspaces,
  \item[(b)] but it does not contain $m$-codimensional proximinal subspaces for any $m\geq n+1$;
  \item[(c)] $\NA(X_n)$ contains $n$-dimensional subspaces,
  \item[(d)] but it does not contain subspaces of dimension greater than $n$.
\end{itemize}
}

Indeed, we consider $X_1$ to be a renorming of $X$ given in \cite[Theorem 4.1]{KLMW-Jussieu} such that $\NA(X_1)$ does not contain $2$-dimensional subspaces. If $X$ is isomorphic to its $n$-power, we just have to follow the proof of Theorem \ref{theorem:main} starting with this $X_1$. Otherwise, for a general $X$ we may consider $X_n=X_1 \oplus_p \ell_p^{(n-1)}$ for $1<p\leq \infty$, which is isomorphic to $X$, and observe that the proof of Theorem \ref{theorem:main} easily adapts to this case.
\end{remark}

\section{Strong proximinality}\label{section_2}
Our next aim is to deal with strong proximinality. First, we show that Read's space does not have strongly proximinal hyperplanes.

\begin{example}\label{example:Read-no-strongprox-hyperplanes}
{\slshape Read space $\mathcal{R}$ does not contain strongly proximinal hyperplanes.}
\end{example}

Let us give the easy argument. By \cite[Lemma 1.1]{GodeInduRevMatComp} (or Lemma \ref{lemma-proximinality}.d), it is enough to check that the norm of $\mathcal{R}^*$ is nowhere strongly subdifferentiable. As $\mathcal{R}^*$ is smooth \cite[Corollary 5]{KadetsLopezMartin}, any point of strong subdifferentiability is actually a point of Fr\'{e}chet differentiability, but the norm of $\mathcal{R}^*$ is $2/3$-rough \cite[Theorem 6]{KadetsLopezMartin}, so it is not Fr\'{e}chet differentiable at any point.

As we commented earlier, there are more spaces with no proximinal subspaces of codimension two which were constructed in \cite{KLMW-Jussieu}. We show that none of them contains strongly proximinal hyperplanes.

\begin{example}\label{example:others-no-strongprox-hyperplanes}
{\slshape None of the spaces constructed in \cite[Proposition 4.4]{KLMW-Jussieu} containing no proximinal subspaces of codimension two, contains strongly proximinal hyperplanes.}
\end{example}

Let $X$ be a space given in \cite[Proposition 4.4]{KLMW-Jussieu}. As for the previous example, it is enough to prove that $X^*$ does not contain points of strong subdifferentiability of the norm. The trick will be to show that every such a point is actually a point of Fr\'{e}chet differentiability of the norm, and then use that the norm of $X^*$ is rough, getting a contradiction. If $X^*$ is smooth, everything is clear, but we only know that when $X^*$ is separable. In the general case, we only know that $X$ is strictly convex. To deal with this case, we just need the following result which follows from \cite{GodeInduRevMatComp} and we state for the sake of completeness.

\begin{lemma}\label{lemma:strictlyconvex-ssd-frechet}
Let $X$ be a strictly convex Banach space. If the norm of $X^*$ is strongly subdifferentiable at a functional $x^*$, then it is Fr\'{e}chet differentiable at $x^*$.
\end{lemma}

\begin{proof}
If follows from \cite[Remarks 1.2.2]{GodeInduRevMatComp} that the set $$A^{**}=\{x^{**}\in S_{X^{**}}\colon x^{**}(x^*)=1\}$$ is contained in the weak-star closure of the set $A=\{x\in S_{X}\colon x^{*}(x)=1\}$. As $X$ is strictly convex, $A$ contains just one point, so the same happens to $A^{**}$. But this means that the norm of $X^*$ is G\^{a}teaux differentiable at $x^*$ so, being strongly subdifferentiable, it is actually Fr\'{e}chet differentiable.
\end{proof}

We do not know whether there is a Banach space containing no proximinal subspaces of codimension two and such that its dual norm is strongly subdifferentiable at any point.

\begin{problem}\label{problem:k=1=n}
Does there exist an infinite-dimensional Banach space having strongly proximinal hyperplanes and containing no proximinal subspace of codimension two?
\end{problem}

We now state a version of Theorem \ref{theorem:main} in which we may control the maximum codimension of a proximinal subspace and the maximum codimension of a strongly proximinal subspace.

\begin{theorem}\label{theorem:strong-n-k}
Let $1< n<\infty$ and $1\leq k < n$. Then there is an infinite-dimensional Banach space $\widetilde X_{n,k}$ such that
\begin{itemize}
  \item[(a)] $\widetilde X_{n,k}$ contains $n$-codimensional proximinal subspaces
  \item[(b)] but it does not contain proximinal subspaces of finite-codimension greater than $n$;
  \item[(c)] $\widetilde X_{n,k}$ contains $k$-codimensional strongly proximinal subspaces,
  \item[(d)] but it does not contain $h$-codimensional strongly proximinal subspaces for $h\geq k+1$.
\end{itemize}
Moreover, for every $\varepsilon>0$ such spaces $\widetilde X_{n,k}$ can be constructed to be isomorphic to $c_0$ and in such a way that $\widetilde X_{n,k}$ and $\widetilde X_{n,k}^*$ are strictly convex and smooth, and $\widetilde X_{n,k}^{**}$ is strictly convex.
\end{theorem}

\begin{proof}
Given $n$ and $k$ as in the hypotheses, $1<p<\infty$, and $0<\varepsilon<2$, consider $$\widetilde X_{n,k}= \ell_p^{(k)}\oplus_p X_{n-k},$$ where $X_{n-k}$ is a space given by Theorem~\ref{theorem:main} which has proximinal subspaces of codimension $n-k$ and $\NA(X_{n-k})$ contains no linear subspaces of dimension higher than $n-k$, which is isomorphic to $c_0$, such that $X_{n-k}^*$ and $X_{n-k}^{**}$ are strictly convex (so $X_{n-k}^*$ is smooth), and such that the norm of $X_{n-k}^*$ is $(2-\varepsilon)$-rough. We first observe that, clearly, $\widetilde X_{n,k}$ is isomorphic to $c_0$, $\widetilde X_{n,k}$ and $\widetilde X_{n,k}^*$ are strictly convex and smooth and that $\widetilde X_{n,k}^{**}$ is strictly convex. Considering a proximinal subspace $Y$ of $X_{n-k}$ of codimension $n-k$, it follows from Lemma \ref{lemma-proximinality}.a that $\{0\}\times Y$ is proximinal in $\widetilde X_{n,k}$ and its codimension is $n=k+(n-k)$. As $\NA(\ell_p^{(k)})=\bigl[\ell_p^{(k)}\bigr]^*$ does not contain linear subspaces of dimension greater than $k$ and $\NA(X_{n-k})$ does not contain linear subspaces of dimension greater than $n-k$, it follows from Corollary \ref{cor-contains_less_d1+d2} that $\NA(\widetilde X_{n,k})$ does not contain linear subspaces of codimension greater than $n$, so $\widetilde X_{n,k}$ does not contain proximinal subspaces of finite-codimension greater than $n$ by Lemma \ref{lemma-proximinality}.c. This gives (a) and (b).

(c). It follows from Lemma \ref{lemma-proximinality}.b that $\{0\}\times X_{n-k}$ is strongly proximinal in $\widetilde X_{n,k}$ and its codimension is $k$.

(d). Consider $(x,y)$ to be a point of strong subdifferentiability of the norm. As the space $\widetilde X_{n,k}^*=\ell_q^{(k)}\oplus_q X_{n-k}^*$ is smooth, $(x,y)$ is actually a point of Fr\'{e}chet differentiability of the norm. It follows from Remark~\ref{remark-Frechet-baja} that if $y\neq 0$, then the norm of $X_{n-k}^*$ if Fr\'{e}chet differentiable at $y$ but, being this norm rough, this is impossible. It then follows that $y=0$. Therefore, if the norm of $\widetilde X_{n,k}^*$ is strongly subdifferentiable at a subspace $W$, then $W\subset \ell_q^{(k)}\times \{0\}$, so $\dim W\leq k$. Finally, Lemma \ref{lemma-proximinality}.d provides that $\widetilde X_{n,k}$ does not contain strongly proximinal subspaces of codimension greater than $k$, as desired.
\end{proof}

We do not know whether the case $k=n$ is possible in the above proposition.

\begin{problem}
Let $1<n<\infty$. Does there exist an infinite-dimensional Banach space having strongly proximinal subspaces of codimension $n$ but no proximinal subspaces of codimension $n+1$?
\end{problem}

A sight to the proof of Theorem \ref{theorem:strong-n-k} shows that a positive answer to Problem \ref{problem:k=1=n} would give a positive answer to the above one.

\section{The infinite-dimensional versions}\label{section_3}
Our final aim in this paper is to present two infinite-dimensional versions of the previous results.

Here is the first version.

\begin{example}\label{example:infinite1}
There exists a non-separable Banach space $X_\infty$ satisfying:
\begin{itemize}
  \item[(a)] $\NA(X_\infty)$ contains infinite-dimensional separable subspaces,
  \item[(b)] but $\NA(X_\infty)$ contains no non-separable subspaces,
  \item[(c)] and every closed subspace contained in $\NA(X_\infty)$ is finite-dimensional.
  \item[(d)] $X_\infty$ contains $n$-codimensional proximinal subspaces for every $n\in \N$,
  \item[(e)] but every proximinal factor reflexive subspace is of finite-codimension,
  \item[(f)] and $X_\infty$ contains no strongly proximinal factor reflexive subspaces (other than $X_\infty$ itself).
\end{itemize}
\end{example}

\begin{proof}
Let $X$ be a non-separable Banach space which is isomorphic to a subspace of $\ell_\infty$ and contains $c_0$, and for every $1<n<\infty$, we consider $Z_n$ to be a renorming of $X$ given in \cite[Theorem 4.1]{KLMW-Jussieu} such that $\NA(Z_n)$ does not contain $2$-dimensional subspaces. We consider the non-separable Banach space $X_\infty=\left[\bigoplus_{n\in \N} Z_n\right]_{c_0}$. Recall that $X_\infty^*=\left[\bigoplus_{n\in \N} Z_n^*\right]_{\ell_1}$ and for each $n\in \N$ we  write $p_n:X_\infty^*\longrightarrow Z_n^*$
to denote the natural projection. By applying Lemma \ref{lemma_NAsum}, it follows that
$$
\NA(X_\infty)\subset \prod\nolimits_{n\in \N} \NA(Z_n)
$$
or, equivalently, that $p_n(\NA(X_\infty))\subset NA(Z_n)$ for every $n\in \N$.
Now, if $W$ is a linear subspace of $X_\infty^*$ which is contained in $\NA(X_\infty)$, we have that $W\subset \prod_{n\in\N} p_n(W)$, and $p_n(W)$ is a linear subspace of $Z_n^*$ contained in $\NA(Z_n)$, so $\dim(p_n(W))\leq 1$. Moreover, if $(f_n)\in \NA(X_\infty)\subset X_\infty^*$, it is immediate to see that the set $\{n\in\N\colon f_n\neq 0\}$ has to be finite. It follows that $\prod_{n\in\N} p_n(W)$, and so $W$, has a countable Hamel basis. Therefore, $W$ is separable, giving (b), and on the other hand, every complete subspace contained in $\NA(X_\infty)$ is of finite-dimension, giving (c). If now $Y$ is a factor reflexive proximinal subspace, then $Y^\perp\subset \NA(X_\infty)$ by Lemma \ref{lemma-proximinality}.c, so $Y^\perp$ and $X_\infty/Y$ are finite-dimensional, giving (e).

Item (f) follows easily from \cite[Theorem 2.2]{InduPIndian2001} and the fact that $Z_n$ has no (non-trivial) strongly proximinal factor reflexive subspaces, but we would like to give a direct proof. Indeed, observe that if $Y$ is a non-trivial factor reflexive strongly proximinal subspace of $X_\infty$, then $Y$ is of finite-codimension by (e) and so Lemma \ref{lemma-proximinality}.d provides the existence of points of strong subdifferentiability of the norm in $X_\infty^*$. But none of the spaces $Z_n^*$ contains points of strong subdifferentiability of the norm by Example \ref{example:others-no-strongprox-hyperplanes} and Lemma \ref{lemma-proximinality}.d, and it follows that $X_\infty^*=\left[\bigoplus_{n\in \N} Z_n^*\right]_{\ell_1}$ contains no point of strong subdifferentiability of the norm by Remark \ref{remark_2}. This contradiction proves (f).

Next, observe that Lemma \ref{lemma-proximinality}.a provides proximinal subspaces of $X_\infty$ of any arbitrary finite codimension, giving (d). Indeed, fix $n\in \N$ and consider $H_k$ a proximinal hyperplane of $X_k$ for $1\leq k\leq n$; then, the $n$-codimensional subspace $$H_1\oplus_\infty \cdots \oplus_\infty H_n \times \left[\bigoplus\nolimits_{k\geq n+1} X_{k}\right]_{c_0}$$
is proximinal in $X_\infty$.

Finally, to get (a), consider for each $n\in \N$ a norm-one $f_n\in \NA(Z_n)$ and observe that the subspace
$$
W:=\bigl\{(t_n f_n)\colon (t_n)\in c_{00} \bigr\}\subset \left[\bigoplus\nolimits_{n\in\N} Z_n^*\right]_{\ell_1}\equiv X_\infty^*
$$
is clearly contained in $\NA(X_\infty)$.
\end{proof}

Let us comment that both $c_0$ and the space of compact linear operators on $\ell_2$ shares the properties (a), (b), (c), (d), and (e) of the space $X_\infty$ constructed above (for $c_0$ is immediate, for the space of compact linear operator on $\ell_2$, see the proof of \cite[Theorem 3.5]{Narayana-Rao}). But they are separable Banach spaces.

By an analogous argument to the one given in Theorem \ref{theorem:strong-n-k} (using Remark \ref{remark_2} instead of Remark \ref{remark-Frechet-baja}), we can easily state the following result.

\begin{remark}
Let $X_\infty$ the space constructed in Example \ref{example:infinite1}. For every $k\in \N$, the non-separable space $\ell_\infty^{(k)}\oplus_\infty X_\infty$ satisfies the conditions (a), (b), (c), (d), and (e) of the cited examples, and also contains strongly proximinal subspaces of codimension $k$, but not higher.
\end{remark}

The second infinite-dimensional version of our results is the following example.

\begin{example}
There exists a non-separable Banach space $\overline{X}_\infty$ satisfying:
\begin{itemize}
  \item[(a)] $\NA(\overline{X}_\infty)$ contains infinite-dimensional closed separable subspaces,
  \item[(b)] but $\NA(\overline{X}_\infty)$ contains no non-separable subspaces.
  \item[(c)] $\overline{X}_\infty$ contains a factor reflexive proximinal subspace such that $\overline{X}_\infty/Y$ is infinite-dimensional and separable,
  \item[(d)] and if $Y$ is a factor reflexive proximinal subspace of $\overline{X}_\infty$, then $\overline{X}_\infty/Y$ is separable.
  \item[(e)] $\overline{X}_\infty$ does not contain strongly proximinal finite-codimensional subspaces (other than $\overline{X}_\infty$ itself).
\end{itemize}
\end{example}

\begin{proof}
Let $X$ be a non-separable Banach space which is isomorphic to a subspace of $\ell_\infty$ and contains $c_0$, and for every $n\in \N$, we consider $Z_n$ to be a renorming of $X$ given in \cite[Proposition 4.4]{KLMW-Jussieu} such that $\NA(Z_n)$ does not contain $2$-dimensional subspaces, that $Z_n$ is strictly convex, and the norm of $Z_n^*$ is rough. Consider the non-separable Banach space $\overline{X}_\infty=\left[\bigoplus_{n\in \N} Z_n\right]_{\ell_2}$. Observe that $\overline{X}_\infty^*=\left[\bigoplus_{n\in \N} Z_n^*\right]_{\ell_2}$, and for each $n\in \N$, write $p_n:X_\infty^*\longrightarrow Z_n^*$
to denote the natural projection. By applying Lemma \ref{lemma_NAsum}, it follows that
$$
\NA(\overline{X}_\infty)\subset \prod\nolimits_{n\in \N} \NA(Z_n)
$$
or, equivalently, that $p_n(\NA(\overline{X}_\infty))\subset NA(Z_n)$ for every $n\in \N$.
Now, if $W$ is a linear subspace contained in $\NA(\overline{X}_\infty)$, then $p_n(W)$ is one-dimensional for every $n\in \N$, as $p_n(W)\subset \NA(Z_n)$. As
$$
W\subset \left[\bigoplus\nolimits_{n\in \N} p_n(W)\right]_{\ell_2},
$$
it follows that $W$ is separable, giving (b). Moreover, if $Y$ is a factor reflexive proximinal subspace of $\overline{X}_\infty$, as $Y^\perp\subset \NA(\overline{X}_\infty)$ by Lemma \ref{lemma-proximinality}.c, we obtain that $Y^\perp$ and $X/Y$ are separable, giving (d).

Next, for each $n\in \N$, we consider a proximinal hyperplane $H_n$ of $Z_n$ and we define $Y=\left[\bigoplus_{n\in \N} H_n\right]_{\ell_2}$. Then, $Y$ is a proximinal subspace of $\overline{X}_\infty$ by \cite[Corollary 4.2]{BandLiLinNarayana} and $\overline{X}_\infty/Y$ is isometric to $\ell_2$. This gives (c). As $Y^\perp \subset \NA(\overline{X}_\infty)$ by Lemma \ref{lemma-proximinality}.c, we also obtain (a).

To get (e), by Lemma \ref{lemma-proximinality}.d, we only have to check that $\overline{X}_\infty^*$ contains no point of strong subdifferentiability of the norm. Indeed, being $\overline{X}_\infty$ an $\ell_2$-sum of strictly convex spaces, it is strictly convex, so if $\overline{X}_\infty^*$ contains a point of strong subdifferentiability of the norm, Lemma \ref{lemma:strictlyconvex-ssd-frechet} gives that it contains a point of Fr\'{e}chet differentiability of the norm. By Remark \ref{remark-Frechet-baja}, this implies that one of the spaces $Z_n^*$ contains a point of Fr\'{e}chet differentiability of the norm, which is impossible since the norm of each $Z_n^*$ is rough.
\end{proof}

We do not know whether the space $\overline{X}_\infty$ contains factor reflexive strongly proximinal subspaces of infinite codimension.

Let us comment that it is immediate $\ell_2$ (or any separable reflexive space) shares the properties (a), (b), (c), and (d) of the space $\overline{X}_\infty$ constructed above. But it is a separable Banach space.

\vspace*{0.5cm}

\noindent \textbf{Acknowledgment:\ } The author thanks Richard Aron for asking the questions which lead to the study of the topics of this manuscript. He also thanks V.~Indumathi for providing the argument to get Lemma~\ref{lemma-proximinality}.d and Gilles Godefroy for kindly answering several inquires related to the topics of this manuscript and provide valuable references.

\end{document}